\documentclass[12pt, oneside]{amsart}
\usepackage{a4wide,amsmath,amssymb,amscd,verbatim}
\input xypic

\newtheorem{prop}[equation]{Proposition}
\newtheorem{thm}[equation]{Theorem}
\newtheorem{cor}[equation]{Corollary}
\newtheorem{lem}[equation]{Lemma}

\theoremstyle{definition}

\newtheorem{rem}[equation]{Remark}
\newtheorem{exa}[equation]{Example}

\numberwithin{equation}{section}

\newcommand{\letbe}{\!\stackrel{\text{def}}{=}}

\newcommand{\hocolim}{\operatorname{hocolim}}

\newcommand{\bF}{\mathbb{F}}

\newcommand{\bQ}{\mathbb{Q}}
\newcommand{\bR}{\mathbb{R}}
\newcommand{\bZ}{\mathbb{Z}}

\newcommand{\cat}[1]{\mbox{\sc #1}}

\newcommand\Z{\bZ}

\newcommand\R{\bR}

\newcommand\Q{\bQ}
\newcommand \fp{\mathbb F_p}
\newcommand \lra{\longrightarrow}

\newcommand \ra{\rightarrow}
\def \larrow#1{\,\stackrel{#1}\lra\,}

\def \CA{\mathcal{A}}

\def \CC{\mathcal{C}}

\def \CF{\mathcal{F}}

\def \CK{\mathcal{K}}

\def \CO{\mathcal{O}}

\def \CU{\mathcal{U}}

\newcommand{\sdcu}{\!-\!\CU}

\newcommand{\catk}{\cat{cat}(K)}

\newcommand{\catcat}{\cat{cat}}

\newcommand{\nz}{\newline}

\newcommand{\catkx}{\catcat(K^\times)}

\newcommand{\op}{^{op}}
\newcommand{\catkxop}{\catkx^{op}}
\newcommand{\Top}{\cat{Top}}

\newcommand{\Alg}{\operatorname{Alg}}

\newcommand{\Dim}{\operatorname{dim}}

\newcommand{\rfgmod}{R_{fg}\!\!-\!\text{mod}}%
\newcommand{\link}{\operatorname{link}}

\newcommand{\Hom}{\operatorname{Hom}}

\newcommand{\Ext}{\operatorname{Ext}}

\newcommand{\depth}{\operatorname{depth}}
\newcommand{\homdim}{\operatorname{hom-dim}}
\newcommand{\st}{\operatorname{st}}

\newcommand{\Tot}{\operatorname{Tot}}
\newcommand{\coker}{\operatorname{coker}}
\newcommand{\id}{\operatorname{id}}
\newcommand{\Gl}{\operatorname{Gl}}

\newcommand{\fpv}{\fp[V]}

\begin{document}
\bibliographystyle{plain}
\vspace*{1cm}

\title{depth and homology decompositions}

\author{Dietrich Notbohm}
\address{Department of Mathematics, Vrije Universiteit Amsterdam, 
De Boolelaan 1081a, 1081 HV Amsterdam}
\email{notbohm@few.vu.nl}


\begin{abstract}
Homology decomposition techniques are
a powerful tool used in the analysis of the homotopy theory of (classifying)
spaces. The associated Bousfield-Kan spectral sequences involve higher
derived limits of the inverse limit functor. We study the impact of
depth conditions on the vanishing of these higher limits and apply our theory
in several cases. We will show that the depth of Stanley-Reisner algebras
can be characterized in combinatorial terms of the underlying
simplicial complexes, the depth of group cohomology in terms of
depth of group cohomology of centralizers of elementary abelian
subgroups, and the depth of polynomial invariants in
terms of depth of
polynomial invariants of point-wise stabilizer subgroups.
The latter two applications follow from the analysis of an algebraic version of centralizer
decompositions in terms of Lannes' $T$-functor.
\end{abstract}

\maketitle

\section{Introduction}

Homology decompositions are one of the most useful tools in the study of  the homotopy theory of topological spaces. Given a cohomology
theory $h^*$,
a homology decomposition for a space $X$
is, roughly speaking, a recipe to glue together spaces,
desirably of a simpler
homotopy type, such that the resulting space maps into $X$
by an $h^*$-isomorphism.

Technically, this is described by a (covariant) functor $F: \CC \lra \Top$
from a (discrete) category $\CC$ into the category $\Top$
of topological spaces together with a natural transformation
$F \lra 1_X$ from $F$ to the constant functor
$1_X$ which maps each object to $X$ and each morphism to the identity.
Passing to homotopy colimits, this natural transformation induces a map $f_X:\hocolim{}_{\CC}F \lra X$.

In special circumstances we can decide by purely algebraic
means, whether the map $f_X$  is an $h^*$-equivalence.
The cohomology of the homotopy colimit can be calculated
with the help of the Bousfield-Kan spectral sequence  \cite{boka}.
The $E_2$-page of this spectral sequence has the form
$$
E_2^{i,j}\, \letbe \lim{}_{\CC^{\op}}^i\, h^*(F)
$$
and converges towards $h^{i+j}(X)$.
If
the derived limits $\lim{}_{\CC^{\op}}^i\, h^*(F)$
of the inverse limit
$\lim{}_{\CC^{\op}}\, h^*(F)$ satisfy the equations
$$
\lim{}_{\CC^{\op}}^i\, h^*(F) \cong
\begin{cases}
h^*(X) & \text{ for } i=0 \\
0 & \text{ for } i\geq 1,
\end{cases}
$$
then $f_X:\hocolim{}_{\CC}\Phi \lra X$ is an $h^*$-equivalence.
The first isomorphism has to be  induced by
the natural transformation
$1_{h^*(X)} \lra h^*(F)$ respectively by the composition
$$
h^*(X) \larrow{h^*(f_X)} h^*(\hocolim{}_{\CC}\,  F)
\lra \lim{}_{\CC^{\op}}\, h^*(F).
$$
We would like to find conditions which ensure the vanishing of
the higher derived limits. In particular,
we are interested in the impact made by depth conditions.

Let $R$ be a  noetherian local  ring with maximal ideal $m_R$
and let
$M$ be a finitely generated $R$-module.
A sequence $x_1,...,x_r\in m_R$ is called regular on $M$,
if for all $i$, the element $x_i$ is not a zero divisor for the
quotient $M/(x_1,...,x_{i-1})M$. And the depth of $M$, denoted
by $\depth M$,
is the maximal length of a regular sequence.

Let $\Phi :\CC \lra \rfgmod$ be a functor, defined on a discrete category
$\CC$ and taking values in the category of
finitely generated $R$-modules. Let $1_M \lra \Phi$
be a natural transformation from the constant functor
$1_M$ to $\Phi$ inducing a $R$-linear map
$\rho:M \lra \lim{}_{\CC}\, \Phi$.
Our technical key result
Lemma \ref{key} relates depth conditions on $M$, depth conditions
 on the values of the functor, on the kernel and cokernel
 of $\rho$, and
the vanishing of the higher limits of $\Phi$.

We will apply Lemma \ref{key} in several cases, namely in the context of
Stanley-Reisner algebras, group cohomology, invariant theory
and algebras over the Steenrod algebra. In fact, with a little extra work,
one could prove that the first three examples are nothing but specialization of the last one. In all cases we will
describe the algebraic object in question as the inverse limit of a
functor defined on a finite category.
The first two application have a topological background in terms of homology decompositions.

\begin{exa} \label{ example1} (Stanley-Reisner algebras)
The main invariant to study the combinatorics of an abstract
simplicial complex $K$ is the associated face ring respectively
Stanley-Reisner algebra $\bF(K)$ ($\bF$ a field),
which is a quotient of a polynomial algebra
generated by the vertices of the complex.
Let $\catkx$ denote the category given by the non-empty faces
of $K$ Then, $\bF(K)$ can be
described as the inverse limit of a functor
$\Phi_K: \catkx \lra \bF\!\!-\!\!\Alg$ from the category $\catkx$
given by the
non-empty faces of $K$ into the category of $\bF$-algebras. In fact, for a
face $\sigma \in K$, we define $\Phi_K(\sigma)\letbe \bF (\st_K(\sigma))$, where
$\st_K(\sigma)$ denotes the star of $\sigma$ (see Section \ref{stanreis}).
This algebraic setting  has a topological realization.
There exists a space $c(K)$ such that $H^*(c(K);\bF)\cong \bF(K)$
as well as a functor $F_K :\catkxop \lra \Top$ defined by
$F_K(\sigma)\letbe c(\st_K(\sigma))$ and a natural transformation
$F_K\lra 1_{c(K)}$ such that
$H^*(F_K ;\bF)\cong \Phi_K$ and such that
$|K|\lra \hocolim_{\catkxop}\, F_K \lra c(K)$ is a cofibration \cite{notopint}(see also Section \ref{stanreis}).

Properties of commutative
algebra for $\bF(K)$
are reflected in the geometry or combinatorics of the underlying
simplicial complex. For example,  Reisner proved such a theorem
for the Cohen-Macaulay property \cite{reisner}. With our methods we are able
to reprove this result. Actually, we will get a slight generalization
and characterize
the depth of the face ring by the cohomological connectivity
of the links of the faces of $K$ (Theorem \ref{mainsr}). This generalization is already at least implicitly contained in Reisner's work (see also
\cite{brun}).
\end{exa}

\begin{exa} \label{example2} (group cohomology)
For a compact Lie group $G$ and a fixed prime p, we denote by
$\CF_p(G)$ the Frobenius or Quillen category of $G$ \cite{quillen}.
The objects are given by
the non-trivial elementary abelian $p$-subgroups $E\subset G$ and the
morphism are all monomorphism $E\lra E'$ induced by conjugation by an element of $G$. For an object $E\in \CF_p(G)$ we denote by $BC_G(E)$
the classifying space
of the centralizer $C_G(E)$ of $E$. These assignments
fit together to establish a functor
$F_G: \CF_p(G)^{\op} \lra \Top$. The inclusions $C_G(E) \subset G$ establish
a natural transformation $F_G \lra 1_{BG}$ and the map
$\hocolim_{\CF_p(G)^{\op}}\, F_G \lra BG$ induces an isomorphism
in mod-$p$ cohomology.
Passing to mod-$p$ cohomology gives rise to a functor
$\Phi_G : \CF_p(G) \lra \fp\!-\!\Alg$ defined by $\Phi_G(E)\letbe H^*(BC_G(E);\fp)$ and to a natural
transformation
$1_{H^*(BG;\fp)} \lra \Phi_G$, which induces an isomorphism
$H^*(BG;\fp) \lra \lim_{\CF_p(G)}\, \Phi_G$.
All these facts
can be found  in \cite{jm} and in \cite{dwcendec}.
We will show that $\depth H^*(BG;\fp)$ equals
the minimum of the numbers
$\depth H^*(BC_G(E);\fp)$ taken over all objects
in $\CF_p(G)$ (Theorem \ref{maingroupcohomology}).
\end{exa}

\begin{exa} \label{example3} (polynomial invariants)
For a finite dimensional  $\fp$-vector space $V$ and a representation
$G\lra \Gl(V)$ of a finite group $G$, we denote by $\fp[V]^G$ the ring of polynomial invariants of the $G$-action on the symmetric
algebra $\fp[V]$ of the dual of $V$.
Let $S$ denote the collection of all point wise stabilizer subgroups
of non trivial subspaces of $V$ whose order is divisible by $p$.
Let $\CO(G)$ denote the orbit category of $G$. That is objects
are given by quotients $G/H$, $H\subset G$ a subgroup,
and the morphisms by $G$-equivariant maps.
Let $\CO_S(G)\subset \CO(G)$ denote the full subcategory of the
orbit category, whose objects are  given by quotients $G/H$ such that
$H\in S$. There exists a functor $\Psi_G : \CO_S(G)^{\op} \lra \fp-\Alg$
and a natural transformation $1_{\fp[V]^G} \lra \Psi_G$.
If $p$ divides the order of $G$, then this natural transformation
establishes an
isomorphism $\fp[V]^G\larrow{\cong} \lim_{\CO_S(G)}\, \Psi_G$
\cite{nodecrep}.
In this case,
we will show that $\depth \fp[V]^G$ equals the minimum of the numbers
$\depth \fp[V]^H$ taken over all subgroups $H\subset G$ contained in $S$
(Theorem \ref{maininv}). If $p$ does not divide the order of $G$,
then $\fp[V]^G$ is Cohen-Macaulay, i.e. $\depth \fp[V]^G$ equals the
dimension of $V$ \cite{nesm}.

The inequality
$\depth \fpv^G\leq \depth \fpv^{H}$ for $H\in S$ was already proven by
Kemper \cite{kemper}
as well as by Smith (see \cite{nesm}).
\end{exa}

\begin{exa} \label{example4} (algebraic centralizer decomposition)
The second and third example and the first in the case of $\bF=\fp$ can be
interpreted as a specialization of a more general statement concerning
algebras over the Steenrod algebra. We can think of this last example
as an algebraic centralizer decomposition for algebras over the
Steenrod algebra.

Let $p$ be a fixed prime. Let $\CA$ denote the mod-p Steenrod
algebra and $\CK$ the category of unstable algebras over $\CA$.
Let $A$ be a noetherian object of $\CK$. i.e. $A$ is noetherian
just as algebra.
Rector constructed a category $\CF(A)$, whose objects $(E,\phi)$
consist of a non-trivial elementary abelian $p$-group $E$ and
a $\CK$-morphism $\phi:A\lra H^E\letbe H^*(BE;\fp)$
into the mod-$p$ cohomology of the classifying space $BE$ of $E$,
such that
$\phi$ makes $H^E$ to a finitely generated $A$-module
\cite{rector}.
With the help of Lannes' $T$-functor \cite{lannes},
one can construct
a functor $\Phi_A : \CF(A) \lra \CK$ and a natural transformation
$1_A\lra \phi_A$ respectively a map
$A \lra \lim{}_{\CF(A)}\, \Phi_A$ (see Section \ref{algcendec}).
We will establish a relation between the depth of $A$ respectively of the
functor values
and the vanishing of the higher limits (Theorem \ref{maincen}).
\end{exa}

Besides Lemma \ref{key} there is one further input in all four examples.
In all four cases it turns out that depth conditions on
$\bF(K)$, $H^*(BG;\fp)$, $\fp[V]^G$ or on $A$ are inherited to the
functor values of the associated functor. This is a consequence of Theorem \ref{depthincrease}, which states that an application of Lannes' $T$-functor
may only increase the depth.

The paper is organized as follows.
In the next section we establish our technical key result. In Section
\ref{depthtfunctor} we analyse the relation between depth and
Lannes' $T$-functor.
This passes the way to
discuss the case of algebraic centralizer decompositions
(Section \ref{algcendec}), the case of group cohomology
(Section \ref{groupcohomology}), the case of polynomial  invariants
(Section \ref{polyinv}) and, finally, the case
of Stanley-Reisner algebras (Section \ref{stanreis}).

\bigskip

\section{The key lemma} \label{keylemma}

Let $\bF$ be a field. To simplify the discussion and since all applications
are covered, we suppose that $R$ is either
a commutative local noetherian ring with maximal ideal $m_R$
and residue field $\bF$ or
or a connected commutative
graded  noetherian $\bF$-algebra. In the graded case, that is that
$R^i=0$ for $i<0$ and $R^0=\bF$. In particular,
$R$ is a also a graded local ring with residue field $\bF$. The
maximal ideal $m_R \letbe R^+ $ is the set of elements of positive degree.
In the graded case, the depth of $M$
is defined as the maximal length
of a sequence of homogeneous elements in $R$ which is regular on $M$.
Let $\rfgmod$
denote the category of finitely generated $R$-modules respectively the category of
finitely generated non-negatively graded $R$-modules.
Let $\Phi :\CC \lra \rfgmod$ be a covariant functor defined
on a discrete category $\CC$.
Let $M$ be an object of $\rfgmod$ and $1_M \lra \Phi$
a natural transformation from the constant functor
$1_M$  to $\Phi$.
This establishes a $R$-linear map $\rho : M \lra \lim_{\CC}\, \Phi$.
We measure the vanishing of the higher derived limits and
the deviation of $\rho$  being an isomorphism by the $R$-modules
$L^i\letbe L^i(M,\Phi)$, $i\geq -1$. That is
$L^{-1}$ is the kernel and $L^0$ the cokernel
of $\rho$. And  $L^i\letbe \lim{}_{\CC}^i\, \Phi_M$ for $i\geq 1$.

Higher derived limits of the functor $\Phi$ can be defined as  the
cohomology groups of a
certain cochain complex
$(C^*(\CC,\Phi), \delta)$ over $R$
(for details see \cite{oliver}).
For $n\geq 0$, the $R$-module $C^n\letbe C^n(\CC,\Phi)$
is  defined as
$$
C^n\letbe
\prod_{c_o\ra c_1 \ra... \ra c_n} F(c_n) \ .
$$
The differential $\delta : C^n \lra C^{n+1}$
is given by the alternating sum
$\sum_{k=0}^r (-1)^k \delta^k$ where $\delta_k$ is
defined on $u\in C^n$ by
$$
\delta^k(u)(c_0\ra\dots\ra c_{n+1})\letbe
\begin{cases}
u(c_0\ra\dots\ra\widehat{c}_k
\ra\dots\ra c_{n+1})&\text{for $k\neq n+1$}\\
\phi(c_n\ra c_{n+1})u(c_0\ra\dots
\ra c_n)&\text{for $k=n+1$}.
\end{cases}
$$
In particular, $\delta^k$ as well as $\delta$ are $R$-linear,
and all the groups $L^i$ are in a natural way finitely generated
$R$-modules.

We fix all the above notation through out this section.
We say that the sequence $L^i$, $i\geq -1$, is almost trivial
if, for all $i\geq -1$, either $L^i=0$ or
$\depth L^i=0$. The following result is the key lemma for proving our main results.

\begin{lem} \label{key}
Let $\Phi : \CC \lra\rfgmod$ be a functor.
Suppose  that $\depth \Phi(c) \geq r$ for all objects $c\in \CC$
and that the sequence $L^i$, $i\geq -1$, is almost trivial.
Then, $\depth M \geq r$ if and only
if $L^i=0$ for $i\leq r-2$.
\end{lem}

The rest of this section is devoted to a proof of this lemma.
The proof is based on spectral sequences associated to
double complexes. To fix notation we recall the basic concept.
Details may be found in \cite{weibel}.

A double complex over $R$ or a differential bigraded
$R$-module $(D^{*,*}, d_h,d_v)$
is a bigraded $R$- module $D^{*,*}$
with two $R$-linear maps
$d_h : D^{*,*} \lra D^{*+1,*}$ and $ d_v:D^{*,*} \lra D^{*,*+1}$ of
bidegree
$(1,0)$ and $(0,1)$ such that $d_hd_h=0=d_vd_v$ and $d_hd_v+d_vd_h=0$.
We think of $d_h$ as the horizontal and of $d_v$ as
the vertical differential.
To each double complex $D^{*,*}$
we associate a total complex $Tot^*(D)$
which is  a differential graded $R$-module defined by
$Tot^n(D)\letbe \oplus_{i+j=n} D^{i,j}$
with differential $d\letbe d_h + d_v$.

If $B_*$ is a chain complex and $C^*$ a cochain complex over $R$
then $\Hom_R(B_*,C^*)$ can be made into a bigraded
$R$-module.
We define $D^{i,j}\letbe \Hom_R(B_{j},C^i)$,
$d_v\letbe \Hom_R(d_B, \id)$ and
$d_h\letbe (-1)^i \Hom_R(\id,d_C)$ for $d_h:D^{i,*} \lra D^{i+1,*}$.

For a double complex $(D^{*,*}, d_h,d_v)$, we can take horizontal or
vertical cohomology groups denoted by
$H^*_h(D^{*,*})$ and $H^*_v(D^{*,*})$.
The boundary maps $d_v$ and $d_h$ induce again boundary maps on
these cohomology groups. We can consider cohomology groups of the form
$H^*_h(H_v^*(D^{*,*}))$ and $H^*_v (H^*_h(D^{*,*}))$.

If $D^{i,j}=0$ for $i<0$ or $j<0$,
there exist two
spectral sequences converging towards $H^*(Tot(D),d)$.
In one case, we have
${}_I E_2^{i,j} = H^i_h(H^j_v(D))$ and in the other case
${}_{II}E_2^{i,j} = H_v^j(H^i_h(D))$.
In the first case the differentials on the $E_r$-page
have degree $(1-r,r)$ and in
the second case degree $(r,1-r)$.

{\it Proof of Lemma \ref{key}:}
Since $R$ is local,
the depth of a finitely generated $R$-module $N$ is characterized
by the smallest number $j$, such that $\Ext_R^j(\bF,N)\neq 0$
\cite[Theorem 4.4.8]{weibel}. Hence, we have to relate certain $\Ext$-groups.

If $r=0$ there is nothing to show since each $R$-module has depth
$\geq 0$. If $r=1$ we have an exact sequence
$$
0\lra \Ext^0_R(\bF,L^{-1}) \lra \Ext_R^0(\bF,M) \lra
\Ext_R^0(\bF,\bigoplus_{c\in C} M(c))
\cong \bigoplus_{c\in C} \Ext^0_R(\bF,M(c))=0.
$$
Hence, since $L^{-1}=0$ or $\depth L^{-1}=0$, we have
$\depth M \geq 1$ if and only if $\depth L^{-1}\geq 1$
if and only if $L^{-1}=0$. This proves the statement for $r=1$.

Now we assume that $\depth M \geq r\geq 2$. In particular,
$L^{-1}=0$. We want to show that $L^i=0$ for $i\leq r-2$.
Let $Q_*$ be a $R$-projective resolution of $\bF$ as $R$-module,
let $C^*\letbe C^*(\CC;\Phi)$ and let $D^{*,*}$ denote the
differential bigraded $R$-module $\Hom_R(Q_*,C^*)$. Since $D^{*,*}$ is concentrated in the first quadrant, both spectral sequences converge
towards $H^*(\Tot(D))$. We have
$$
{}_I E_2^{i,j} =H^i_h(H^j_v(D))\cong
\lim{}^i\, \Ext^j_R(\bF,\Phi),
$$
which vanishes for $j\leq r-1$.
On the other hand,
we have
$$
{}_{II} E_2^{i,j}= H^j_v(H^i_h(D))\cong
\begin{cases}
\Ext^j_R(\bF,L^i) & \text{ for } i\geq 1 \\
\Ext^j_R(\bF,\lim{}^0_{\CC}\, \Phi) & \text{ for } i=0.
\end{cases}
$$
The long exact sequence of $\Ext$-groups
for the short exact sequence
$$
0\lra M \lra \lim{}_{\CC}\, \Phi \lra L^0 \lra 0
$$
shows that $\Ext^i_R(\bF,\lim{}_{\CC}\, \Phi)\cong
\Ext^i_R(\bF,L^0)$ for $i\leq r-2$.

By degree reasons and since no differential starts or ends at
${}_{II} E_2^{0,0}\cong \Ext^0_R(\bF,L^0)$, this group has
to vanish. Since $L^{0}=0$ or $\depth L^0=0$, this shows that $L^0=0$,
that $M\cong \lim{}^0_{\CC}\, \Phi$ and that ${}_{II} E_2^{0,j}=0$
for $j<r$. We can repeat the argument successively for $i=1...,r-2$,
which implies that $L^i=0$ for $i\leq r-2$.
Since there might be a non-trivial differential
$$
{}_{II} E_2^{r-1,0}\cong \Ext^0(\bF,L^{r-1}) \lra
{}_{II} E_2^{0,r}\cong \Ext^r_R(\bF,M),
$$
we cannot repeat the argument further.

If $L^i=0$ for $i\leq r-2$,
the $E_2$-page of first spectral sequence is exactly the same,
but may fail the vanishing statement.
Since $M\cong \lim{}^0_{\CC}\, \Phi$, the $E_2$-page of the second spectral sequence satisfies the equations
$$
{}_{II} E_2^{i,j}\cong
\begin{cases}
\Ext^j_R(\bF,M) & \text{ for } i=0\\
0 & \text{ for } 1\leq i\leq r-2.
\end{cases}
$$
A similar degree argument shows that
$\Ext^j_R(\bF,M)=0$ for $j\leq r-1$ and that $\depth M\geq r$. This finishes the proof.
\qed

\begin{rem}
Lemma \ref{key} can be slighly generalized. If you assume that
either $L^i=0 $ or $\depth L^i=s$ for all $i$, then a modification of the above argument will show that
$\depth M \geq r$ if and only if $L^i=0$ for $i\leq r-s-2$.
\end{rem}

We also will apply Lemma \ref{key} in form of the following corollary.

\begin{cor} \label{litrivial}
If $L^i=0$ for all $i\geq -1$ then
$\depth M \geq \min\{\depth \Phi_M(c): c\in \CC\}$.
\end{cor}

\bigskip
\section{Depth and the $T$-functor} \label{depthtfunctor}

To fix notation we recall some basic concepts from
$T$-functor technology. For more details see
\cite{schwartz} and \cite{lannes}

Let $p$ be a fixed prime. As already mentioned in the introduction,
$\CA$ denotes the mod-$p$ Steenrod algebra,
$\CU$ the category of unstable modules and $\CK$
the category of unstable algebras
over $\CA$.

For an elementary abelian $p$-group $E$
we denote by
$H^E\letbe H^*(BE;\fp)$ the mod-$p$ cohomology of
the classifying space $BE$.
The functor $\otimes H^E : \CU \lra \CU$ has a left
adjoint in $\CU$ denoted by
$T_E:\CU\lra \CU$. This functor is exact,
commutes with tensor products and
restricts therefore to a functor $T_E:\CK \lra \CK$
also denoted by $T_E$.
The adjoint of $\id : T_E(M) \lra T_E(M)$ establishes a map
$M\lra T_E(M)\otimes H^E$ and, projecting to the first
factor via the
augmentation $H^E \lra \fp$, a map
$e: M\lra T_E(M)$.

We can specialize further.
For an algebra $A\in \CK$
we denote by  $A\sdcu$ the
category of $A\odot \CA$-modules. An $A\odot \CA$-module is both,
an object $M\in \CU$ and an $A$-module such that the structure map
$A\otimes M\lra M$ is $\CA$-linear.  Here, $\CA$ acts
on the tensor product via the Cartan formula.
We also denote by $A_{fg}\sdcu\subset A\sdcu$ the full subcategory of
all objects which are finitely generated as $A$-modules.

If $M\in A\sdcu$,
then, since the functor $T_E$ commutes with tensor products, the
$\CU$-object $T_E(M)$ is actually an object of $T_E(A)\sdcu$.
And if $\phi : A \lra H^E$ is a $\CK$-map,  we define
$M(\phi)\letbe T_E(M;\phi)\letbe T_E(M)\otimes_{T^0_E(A)}\fp(\phi)$,
where $\fp(\phi)$ denotes the $T_E(A)$-module whose structure map
is induced by the adjoint of the $\CK$-map $\phi : A\lra H^E$.

Let $F(n)\in \CU$ denote the free object in $\CU$ generated by one element
in degree $n$. The module $F(n)$ is defined by the equation
$\Hom_{\CU}(F(n),M)\letbe M^n$ for any object $M\in \CU$.
Here, $M^n$ denotes the set of elements of $M$ of degree $n$
(for details see \cite{schwartz}).
For any object $M\in A\sdcu$, we have
$\Hom_{A\sdcu}(A\otimes_{\fp}F(n), M)=Hom_{\CU}(F(n),M)=M^n$.
Therefore, there exists an $A\sdcu$-epimorphism $P\lra M$ such that
$P$ is free as $A$-module as well as a resolution
$$
\lra P_n \lra P_{n-1}\lra ...\lra P_0\lra M \lra 0
$$
of $A\sdcu$-modules, such that each $P_i$ is free as $A$-module.

In the following we assume that $A$ is
noetherian as an algebra, which
simplifies the discussion. In particular,
$\Hom_{\CK}(A, H^E)$ is a finite set and
$T^0_E(A)\cong \fp^{\Hom_{\CK}(A,H^E)}$ is a finite dimensional
vector space isomorphic to the dual of $\Hom_{\CU}(A,H^E)$.
And, if $M\in A\sdcu$, then
$$
T_E(M)\cong T_E(M)\otimes_{T^0(A)}T^0(A) \cong
\bigoplus_{\phi \in \Hom_{\CK}(A,H^E)} M(\phi).
$$
The above map $e:M \lra T_E(M)$ respects all these additional structures
and, restricting to  a particular summand of $T_E(M)$,
establishes a map
$e_\phi: M \lra M(\phi)$. This is a map of $A$-modules,
where $A$ acts on $M(\phi)$ via
the $\CK$-map $A\lra A(\phi)$, and therefore a
$A\sdcu$-morphism. In fact, if $A\in \CK$ is noetherian and $A\lra H^E$ a
$\CK$-map, then the functor $T_E(-,\phi)$ restricts to a functor
$T_E(-,\phi): A_{fg}\sdcu \lra A_{fg}\sdcu$ \cite{henn}. In particular, if $A$ is noetherian, so is $A(\phi)\letbe T_E(A,\phi)$.

In this section we want to prove the following theorem.

\begin{thm} \label{depthincrease}

Let $A\in \CK$ be a noetherian connected algebra
and $M\in A_{fg}\sdcu$. Let $E$ be an elementary
abelian group and let $\phi : A\lra H^E$ be a $\CK$-map.
Then,
$\depth M\leq \depth M(\phi)$.
\end{thm}

The proof of this theorem is very similar to an argument used by Smith
in a different context (see \cite{nesm}).

Using the Auslander Buchsbaum equation (e.g see \cite{weibel})
we will reduce this statement to a claim about homological dimension.

Let $A$ be a connected commutative graded algebra and $M$ a graded $A$-module.
We say that $M$ has finite homological dimension if there exists a
finite resolution
$$
0 \lra P_r \lra P_{r-1} \lra ... \lra P_0 \lra M \lra 0
$$
of graded $A$-modules, such that $P_i$ is projective for $0\leq i \leq r$.
We say that $\homdim_A M=t$ if the shortest projective resolution has length $t$.

\begin{prop} \label{homdimdecrease}
Let $A$ be a noetherian object of  $\CK$, $M$ and object of $A\sdcu$, and
$\phi :A\lra H^E$ a $\CK$-map.
If $M$ has finite homological dimension over $A$,
then $M(\phi)$ has finite homological dimension over $A(\phi)$and
$\homdim_A M \geq \homdim_{A(\phi)} M(\phi)$.
\end{prop}

\begin{proof}
Let $\homdim_A M=t$. We can choose a resolution
$$
0\lra P_t \lra P_{t-1}\lra ...\lra P_0 \lra M \lra 0
$$
of $A\sdcu$-moduls, such that $P_i$ is free as
$A$-module for $0\leq i\leq t-1$.
Since $\homdim_A M=t$, this implies that $P_t$ is a projective
$A$-module. And since $A$ is  connected and noetherian, in particular a graded
local ring ring, $P_t$ is a free $A$-module.
%
%

Applying the exact  functor $T_E(-,\phi)$ establishes an
exact $A(\phi)\sdcu$-sequence
$$
0\lra P_t(\phi) \lra P_{t-1}(\phi) \lra ... \lra P_0(\phi) \lra M(\phi)\lra 0.
$$
Lannes $T$-functor maintains freeness properties \cite{dwkaehler}.
Hence, for all $i$, the module $P_i(\phi)$ is a free $A(\phi)$-module,
which proves the claim.
\end{proof}

The proof of Theorem \ref{depthincrease}  needs some further preparation.
Let $\alpha : A\lra B$ and $\phi: A\lra H^E$  be  $\CK$-maps,
$A$ and $B$ noetherian.
We denote by $e(\phi,\alpha)$ the set of all
$\CK$ maps $\psi:B\lra H^E$ such that $\psi\alpha=\phi$.
This is a finite set.

\begin{lem} \label{algextension}
Let $\alpha : A\lra B$ and $\phi : A\lra H^E$ as above. For each
object $M$ of $B\sdcu$, the $T$-functor induces an $A\sdcu$-isomorphism
$M(\phi)\cong \bigoplus_{\psi\in e(\phi,\alpha)} M(\psi)$.
\end{lem}

\begin{proof}
The algebra $A$ acts on $B(\psi)$ via the composition
$A\lra B \lra B(\psi)$ of $\CK$-maps.
Since
$T^0(B)\otimes_{T^0(A)}\fp(\phi)
\cong \bigoplus_{\psi\in e(\phi,\alpha)}\fp(\psi)$
as algebras, we have
$$
\begin{array}{rcl}
M(\phi)& \cong & T(M)\otimes_{T^0(A)}\fp(\phi)\\
& \cong & T(M)\otimes_{T^0(B)}T^0(B)\otimes_{T^0(A)}\fp(\phi) \\
& \cong & \bigoplus_{\psi \in e(\phi,\alpha)} T(M)\otimes_{T^0(B)}\fp(\psi)\\
& \cong & \bigoplus_{\psi \in e(\phi,\alpha)} M(\psi).
\end{array}
$$
And this composition is obviously a map of $A$-modules.
\end{proof}

The Dickson algebra
$D_r\cong \fp[d_1,...d_r]\cong \fp[x_1,...,x_r]^{\Gl(r,\fp)}$
is a polynomial algebra isomorphic to the invariants of
the  general linear group $\Gl(r,\fp)$ acting on the polynomial algebra
$\fp[x_1,...,x_r]$ with $r$-generators, which we
give the degree 2.
Then, this algebra carries an action of $\CA$ which inherits an action to $D_r$.  More details can be found in
\cite{nesm}.

For a positive integer $l$
we denote by $D_r^l\subset D_r$ the subalgebra of all $p^l$-powers of
elements of $D_r$. Then, $D^l_r\cong \fp[d_1^{p^l},...,d_r^{p^l}]$
is again a polynomial algebra , in particular a noetherian object
of $\CK$ \cite{schwartz}.

{\it Proof of Theorem \ref{depthincrease}:}
Since $A$ is noetherian, there exist integers $n,l\geq 1$
and a $\CK$ map $\alpha: D\letbe D_n^l \lra A$ making $A$ into a finitely generated $D$-module \cite[Appendix]{boza}.
Hence, $M$ is a finitely generated $D$-module as well,
and for the calculation of $\depth M$ we can consider $M$ as an $D$-module
\cite{serre}.
On the other hand, for each $\CK$-map $\phi: A\lra H^E$,
we have
$M(\phi\alpha)\cong \bigoplus_{\psi\in e(\phi\alpha,\alpha)} M(\psi)$
as $D$-modules.
Since $\phi\in e(\phi\alpha,\alpha)$, we have
$\depth M(\phi)\geq \depth M(\phi\alpha)$. Hence,
we can and will assume
that $A=D$.

Since $D$ is a polynomial algebra, every finitely generated $D$-module
has finite homological dimension (see \cite{weibel}). By
Proposition \ref{homdimdecrease} we have
$\homdim_D M\geq \homdim_{D(\phi)} M(\phi)$.
Since $D$ is a polynomial algebra, the same
holds for $D(\phi)$ and  $D(\phi)$ is a free  finitely generated $D$-module
\cite{dwkaehler}. In particular, every projective $D(\phi)$-module is a projective $D$-module and
$\homdim_{D(\phi)} M(\phi)\geq \homdim_D M(\phi).$
By the Auslander-Buchsbaum equation (see \cite{weibel}), we have
$\depth N=\depth D -\homdim_D N$ for any finitely generated $D$-module $N$.
Proposition \ref{homdimdecrease}
shows that $\depth M\leq \depth M(\phi)$.
\qed.

\bigskip

\section{Depth and the algebraic centralizer decomposition}
\label{algcendec}

We use the same notation as in the last section. For a
noetherian object $A\in \CK$, Rector defined a category
$\CF(A)$ as follows \cite{rector}.
The objects are given by pairs $(E,\phi)$, where
$E$ is a non-trivial elementary abelian $p$-group and where
$\phi : A \lra H^E$ is a $\CK$-map such that
$H^E$ becomes a finitely generated
$A$-module. And a morphism $\alpha : (E,\phi)\lra (E',\phi')$
is a monomorphism
$i_\alpha: E \lra E'$ such that $\phi H^*(Bi_\alpha)=\phi'$.
Since $A$ is noetherian, this category is finite.

Let $M$ be an object of $A\sdcu$.
Since the $T$-functor its natural with respect to homomorphism
$E\lra E'$, the maps $e_\phi : M\lra M(\phi)$ are compatible with
all morphisms in the
category $\CF(A)$. This defines a covariant functor
$$
\Phi_M : \CF(A) \lra A\sdcu: M \mapsto M(\phi)
$$
as well as a natural transformation $1_M\lra \Phi_M$.
In particular, there exists  an $A\sdcu$-morphism
$$
\rho_{M} : M\lra \lim{}_{\CF(A)}\, M(\phi).
$$
For a topological interpretation of these algebraic data and
their relation to centralizer decompositions of compact Lie groups
see \cite{dwcendec} (see also Section \ref{groupcohomology}).

As in Section \ref{keylemma}, we define
$$
L^i(M)\letbe
\begin{cases}
\lim{}^i_{\CF(A)}\, \Phi_M & \text{ for } i\geq 1\\
\coker(\rho_M) & \text{ for } i=0\\
\ker(\rho_M) & \text{ for } i=-1
\end{cases}
$$

\begin{thm} \label{maincen}
Let $A\in \CK$ be a connected noetherian algebra and
let $M$ be an object of $A_{fg}\sdcu$.
 Then,
$\depth M\geq r$
if and only if the following two conditions hold:
\nz
(i) For all objects $(E,\phi)\in \CF(A)$, we have
$\depth M(\phi)\geq r$.
\nz
(ii)
$L^i(M)=0$ for $i\leq r-2$
\end{thm}

\begin{proof}
In \cite{henn}, Henn has shown that the $A$-modules $L^i$ are finite dimensional
$\fp$-vector spaces. In particular, either $L^i(M)=0$ or
$\depth L^i=0$ for all
$i\geq -1$.
That is, the sequence $\{L^i: i\geq -1\}$ is almost trivial. Moreover, by Theorem \ref{depthincrease},
we have $\depth M(\phi) \geq \depth M$. An application of
Lemma \ref{key} finishes the proof.
\end{proof}
\bigskip
\section{Depth and group cohomology} \label{groupcohomology}

Let $G$ be a compact Lie group with classifying space $BG$. And
let $p$ be a fixed prime. In this section,
cohomology is always taken  with $\fp$-coefficients
and $H^*(-)\letbe H^*(-;\fp)$.

As in the introduction we denote by $\CF_p(G)$ the Quillen category of
$G$ and by
\nz
$F_G : \CF_p(G)^{\op} \lra \Top$ the functor mapping
$E$ to $BC_G(E)$. The natural transformation $F_G \lra 1_{BG}$
establishes a map
$\hocolim_{\CF_p(G)^{\op}}\,  F_G \lra BG$.
Jackowski and McClure showed that this map induces
an isomorphism in mod-$p$ cohomology.
In fact they showed that
$H^*(BG)\cong \lim{}_{\CF_p(G)}\, H^*(F_G)$,
that $\lim{}^i_{\CF_p(G)}\, H^*(F_G)=0$ for $i\geq 1$ \cite{jm}.

Passing to classifying spaces and mod-p cohomology
induces a functor
$$
\CF_p(G)^{\op} \lra \CF(H^*(BG)),
$$
which turns out to be a an equivalence of categories \cite{dwcendec}.
Moreover, for a subgroup $i_e : E\subset G$, we have
$H^*(BC_G(E)\cong T_E(H^*(BG),H^*(i_E))$ \cite{lannes}.
Hence, identifying both categories,
there exists a natural equivalence
$H^*(F_G) \larrow{\cong} \Phi_{G}$, where
$\Phi_G\letbe \Phi_{H^*(BG)} : \CF(H^*(BG)) \lra H^*(BG)\sdcu$ denotes
the functor constructed
in the previous section in the case $M=A=H^*(BG)$.
In particular, this implies that
$L^i(H^*(BG))=0$ for all $i\geq -1$.
and that $\depth H^*(BG) \leq \depth H^*(BC_G(E)$
(Theorem \ref{depthincrease}).

The inclusion $C_G(E) \subset G$ induces a $\CK$-map
$H^*(BG) \lra H^*(BC_G(E))$ and makes the target into
a finitely generated $H^*(BG)$-module \cite{quillen}.
Hence,
$\depth_{H^*(BG)} H^*(BC_G(E))=\depth_{H^*(BC_G(E))} H^*(BC_G(E))$
\cite{serre}).

Applying Corollary \ref{litrivial} and
fitting all the above arguments together, this proves the following
statement.

\begin{thm} \label{maingroupcohomology}
Let $G$ be a compact Lie group.
Then,
$$
\depth H^*(BG)=\min\{\depth H^*(BC_G(E)) : E\in \CF_p(G)\}.
$$
\end{thm}

If $E\cong E'\times E''$ then we have $E\subset C_G(E')$ and
$C_G(E)=C_{C_G(E')}(E)$. Hence, in the above theorem, we only have to take the minimum over all 1-dimensional elementary abelian $p$-subgroups.

\begin{rem}
Results of the Jackowski-McClure type do exist for homotopy theoretic
versions of groups. In particular, Theorem
\ref{maingroupcohomology} also holds for $p$-compact groups and
p-local finite groups. For definitions and concepts of these notions
see \cite{dwannals} and \cite{dwcenter} respectively \cite{blo1}.
\end{rem}

\bigskip

\section{Depth and polynomial invariants} \label{polyinv}

We use the same notation as in the introduction,
$V$ denotes a finite dimensional  $\fp$-vector space,
$G\lra \Gl(V)$ is a (faithful) representation,
$\fp[V]$ denotes the ring of polynomial functions
on $V$ respectively the symmetric
algebra on the dual
of $V$, $\fp[V]^G\subset \fp[V]$
the subalgebra of polynomial invariants,
$S$ the collection of all point wise stabilizer subgroups
of non-trivial subspaces of $V$ whose order is divisible by $p$,
$\CO_S(G) \subset \CO(G)$
the full subcategory of the orbit category associated to the collection $S$
and $\Psi_G : \CO_S \lra \fp[V]^G\!-\!\!\mod$ the functor given by
$G/H\mapsto \fp[V]^H$.  For a linear subspace $U\subset V$, we
denote by $G_U$ the point wise stabilizer of $U$.

\begin{thm} \label{maininv}
If $p$ divides the order of $G$ and $G\lra \Gl(V)$ is a faithful
representation, then
$$
\depth \fp[V]^G = \min\{\depth \fpv^{G_U} : 0\neq U \subset V\}.
$$
\end{thm}

In particular, this says that $\depth \fpv^G \leq \depth \fpv^{G_U}$ for all non-trivial subspaces $U\subset V$. Since, for
$U\subset W \subset V$,
we have $G_U=(G_W)_U$,
we get away in the above statement by
taking the minimum over all 1-dimensional subspaces.

\begin{proof}
The proof is based on Lemma \ref{key} and Theorem \ref{depthincrease}.
Before we can apply these results, we have to recall the necessary setting.
We make $\fpv$ into a graded  $\fp$-algebra by giving  all polynomial
generators the degree 2. Then, there exists a unique $\CA$-action on $\fpv$
and $\fpv$ becomes an object of $\CK$. Moreover, the $\CA$-action is inherited to the ring of polynomial invariants $\fpv^G$.
If $U\subset V$, the composition
$\phi : \fpv^G \lra \fpv \lra \fp[U]\lra H^U$
is a $\CK$ map and $T_U(\fpv^G,\phi)\cong \fpv^{G_U}$. For example, all these
constructions and claims can be found in
\cite{nesm}. In particular, this shows that
$\depth \fpv^G\leq \depth \fpv^{G_U}$ for all non-trivial subspaces
$U\subset V$ (Theorem \ref{depthincrease}).

By \cite{nodecrep} we have
$$
\lim{}^i_{\CO_S(G)}\, \Phi_G \cong
\begin{cases}
\fpv^G & \text{ for } i=0\\
0 & \text{ for } i\geq 1.
\end{cases}
$$
Since $\depth \fpv^G\leq \depth \Phi_G(G/H)=\fpv^{H}$  for all objects
$G/H\in \CO_S(G)$, an application of Corollary \ref{litrivial} finishes the proof.
\end{proof}

\bigskip

\section{Stanley-Reisner algebras} \label{stanreis}

An abstract simplicial complex $K$ with $m$-vertices given by the
set $V\letbe \{1,...,m\}$ consists of a finite set
$K\letbe \{\sigma_1,...,\sigma_r\}$ of subsets of $V$, which  is closed
with respect to  formation of subsets. The subsets $\sigma_i \subset V$ are called the faces of $K$.
The dimension
of $K$ is denoted by $\Dim K$ and $\Dim K = n-1$, if every
face $\sigma$ of $K$ has
order $\sharp\sigma\leq n$ and there exists a maximal face $\mu$ of
order $\sharp\mu=n$. We consider the empty set $\emptyset$
as a face of $K$.

For a field $\bF$ we denote by $\bF(K)$
the associated face ring or
Stanley-Reisner algebra
of $K$ over $\bF$. It is the quotient
$\bF[V]/(v_\sigma : \sigma \not\in K)$, where
$\bF[V]\letbe \bF[v_1,...,v_m]$ is a polynomial algebra
on $m$-generators
and $v_\sigma\letbe  \prod_{j\in \sigma} v_j$.
We can think of $\bF(K)$ as a
graded object. It will be convenient to choose
the topological grading and  give
the generators of $\bF(K)$ and $\bF[V]$ the degree 2.

Each abstract simplicial complex $K$ has a geometric realization,
denoted by $|K|$.
%
%
We define the cohomology groups $H^*(K)$ of $K$ as
the cohomology groups $H^*(|K|)$ of the
topological realization. And $\widetilde H^*(K)$ will denote the
reduced cohomology.

For a face $\sigma\in K$, the star $\st_K(\sigma)$ and the link
$\link_K(\sigma)$ of $\sigma$
are defined as the simplicial subcomplexes
$\st_K(\sigma)\letbe \{\tau\in K: \sigma\cup \tau \in K\}$ and
$\link_K(\sigma)\letbe \{\tau\in K :
\sigma\cup\tau \in K, \sigma\cap \tau=\emptyset\}$.

\begin{thm} \label{mainsr}
Let $K$ be an abstract finite simplicial complex. Let
$r\leq \Dim K+1$. Then the following conditions are equivalent:
\nz
(i) $\depth \bF(K) \geq r$.
\nz
(ii) For all
faces $\sigma\in K$ we have $\widetilde H^i(\link_K(\sigma);\bF)=0$
for $i\leq r-\sharp\sigma-2$.
\nz
(iii) $\widetilde H^i(K;\bF)=0=H^i(|K|,|K|-x;\bF)$
for $i\leq r-2$ and for all
points $x\in |K|$.
\end{thm}

The equivalence of the last two conditions was shown
by Munkres \cite{munkres}.
In fact, for
each face $\sigma\in K$ and each inner point
$x\in \sigma$, he showed that
$\widetilde H^{i-\sharp\sigma}(\link_K(\sigma))\cong H^i(|K|,|K|-x)$.
This shows  that the depth condition
only depends on the topological structure
of $|K|$,
and not on the
simplicial structure of $K$.

Since the Krull dimension of $\bF(K)$ equals $\Dim K +1 =n$, the face ring
$\bF(K)$ is Cohen-Macaulay if and only if $\depth \bF(K)=n$.
This shows that the above theorem generalizes Reisner's characterization
of Cohen-Macaulay face rings \cite{reisner}.

The proof needs some preparation.
The poset structure of $K$, given by the subset relation between the
faces, gives rise to a category denoted by $\catk$, and
the non-empty faces
of $K$ to the full subcategory $\catkx \subset \catk$.
There exists a functor $\Phi_K :\catkx \lra \Alg_{\bF}$
taking values in the category of $\bF$-algebras.
For a face $\sigma\in K$ we define
$\Phi_K(\sigma)\letbe \bF(\st(\sigma))$. For an inclusion $\sigma \subset \tau$, we have $\st_K(\tau) \subset \st_K(\sigma)$, which
defines the map $\bF(\st_K(\sigma)) \lra \bF(\st(\tau))$.

The following theorem is proved in \cite{notopint}:

\begin{thm} \label{srdec}
Let $K$ be an abstract finite simplicial complex.
Then,
$$
\lim{}^i_{\catkx}\, \Phi_K \cong
\begin{cases}
\bF(K) \oplus \widetilde H^0(K;\bF) & \text{ for } i=0 \\
H^i(K;\bF) & \text{ for } i\geq 1.
\end{cases}
$$
\end{thm}

\begin{prop} \label{depthstarinherited}
Let $K$ be an
abstract finite simplicial complex. For all faces
$\sigma \in K$ we have
$\depth \bF( \link_K(\sigma)) + \sharp \sigma
= \depth \bF(\st_K(\sigma)
\geq \depth \bF(K)$.
\end{prop}

\begin{proof}
Since $\bF(\st_K(\sigma))\cong \bF[\sigma]\otimes \bF(\link_K(\sigma))$,
the first equation always holds. Here, $\bF[\sigma]\subset \bF[V]$
denotes the polynomial subalgebra generated by all $v_i$ such that
$i\in \sigma$.

For the proof of the second equation
we first assume that $\bF=\fp$. Then, $\fp(K)$ is an unstable algebra
over the Steenrod algebra. For  $\sigma\subset V$ we denote by
$E^\sigma$
the
$\sharp \sigma$-fold product of the cyclic group $\Z/p$ of order $p$.
Then, $\fp[\sigma]\subset H^\sigma\letbe H^{E^\sigma}$ is a subalgebra  in the category $\CK$. In fact, it is the polynomial part of $H^\sigma$.
Let $\phi_\sigma$ denote the composition
$\phi_\sigma: \fp(K) \lra \fp[\sigma]\lra H^\sigma$.
By \cite{nora2}, there exists an isomorphism
$T_{E^\sigma}(\fp(K),\phi_\sigma)\cong \fp(\st_K(\sigma))$.
Now, the second equation follows from
Theorem \ref{depthincrease}.

If $\bF$ is a general field of characteristique $p>0$,
the claim follows from the observation that
$\depth \fp(K)=\depth \fp(K)\otimes_{\fp} \bF=\depth \bF(K)$.

If $\bF=\Q$, then we notice that $\depth \Q(K)$
is the maximum of the set $\{\depth \fp(K) :p \text{ a prime }\}$.
And if $\bF$ is a general field of characteristique 0, then we can argue as in
the case of positive characteristique.
\end{proof}

{\it Proof of Theorem \ref{mainsr}:}
By Theorem \ref{srdec}, the groups $L^i\letbe L^i(\bF(K),\Phi_K)$
are all finite and vanish for $i=-1,$ and $i\geq \dim K +1$.
We can apply Lemma \ref{key}.

If $\depth \bF(K) \geq r$,
then $\depth \bF(\st_K(\sigma))=\depth \bF(\Phi(\sigma))\geq r$
for all $\sigma\in K$
(Proposition \ref{depthstarinherited}), and
$\widetilde H^i(K;\bF)=L^i(\catkx;\Phi_K)=0$ for $i\leq r-2.$
(Lemma \ref{key}).

If $\depth \bF(\st_K(\sigma))\geq r$ for all
$\emptyset \neq\sigma \in K$ and if $\widetilde H^i(K;\bF)=0$
for $i\leq r-2$, Lemma \ref{key} shows that $\depth \bF(K)\geq r$.

Since $\depth \bF(\link_K(\sigma))=\depth \bF(\st_K(\sigma))-\sharp\sigma$
and since $\Dim \link_K(\sigma)<\Dim K$, an induction over the
dimension of the simplicial complexes proves the claim.
\qed

%
%
%
%
%
%
%
%
%
%
%
%
%
%
%
%
%




\begin{thebibliography}{600}

\bibitem[1]{boza} D. Bourguiba and S. Zarati,
\emph{Depth and Steenrod Operations}, Inv. Math. 128 (1997),589-602.

\bibitem[2]{boka} A.K. Bousfield and D.M. Kan,
\emph{Homotopy limits, completions and localisations}, SLNM 304,
Springer Verlag.

\bibitem[3]{brun} M. Brun, W Bruns and T. R\"omer, \emph{Cohomology
of partially ordered sets and local cohomology of section rings},
preprint, math.AC/0502517.

%
\bibitem[3]{blo1} C. Broto, R. Levi and R. Oliver,
\emph{The homotopy theory of fusion systems},
J. Amer. Math. Soc. 16 (2003), 779-856.

\bibitem[4]{bupa}
Victor~M Buchstaber and Taras~E Panov.
{\em Torus Actions and Their Applications in Topology and
  Combinatorics}, volume~24 of {University Lecture Series},
American Mathematical Society, 2002.

\bibitem[5]{dwcendec} W.G. Dwyer and C.W. Wilkerson,
\emph{A cohomology decomposition theorem},
Topology 31 (1992), 29-45.

\bibitem[6]{dwannals} W.G. Dwyer and C.W. Wilkerson,
\emph{Homotopy fixed point methods for Lie groups and finite loop spaces},
Annals of Math. 139 (1994), 395-442.

\bibitem[7]{dwcenter} W.G. Dwyer and C.W. Wilkerson,
\emph{The center of a p-compact group}, Contemp. Math. 181, Proc. 1993 Conf.
AMS 1995, 119-157.

\bibitem[8]{dwkaehler} W.G. Dwyer and C.W. Wilkerson,
\emph{K\"ahler differentails, the $T$-functor, and a theorem of Steinberg}
Trans. AMS 350 (1998), 4919-4930.

\bibitem[9]{henn} H.-W. Henn, \emph{Commutative algebra of
unstable $K$-modules, Lannes' $T$-functor and equivariant mod-$p$
cohomology}, J.Reine Angew. Math. 478 (1996), 189-215.

\bibitem[10]{jm} S. Jackowski and J. McClure, \emph{Homotopy
decomposition of classifying spaces vis elementary abelian subgroups}, Toplogy 31 (1992), 113-132.

\bibitem[11]{kemper} Kemper, \emph{Loci in quotients by finite groups,
pointwise stabilzers and the Buchsbaum property}, J. Reine Angew. Math. 547
(2002), 69-96.

\bibitem[12]{lannes} J. Lannes, \emph{Sur la cohomologie modulo p de
p-groupes ab\'eliens \'el\'ementaire}, Publ. Math. IHES 75 (1992),
135-244.

\bibitem[13]{munkres} J.R. Munkres, \emph{Toplogical results in
combinatorics}, Michigan Math. J. 31 (1984), 113-128.

\bibitem[14]{nesm} M. Neusel and L. Smith, \emph{Invariants theory
of finite groups}, AMS 2002.

\bibitem[15]{notopint} D. Notbohm, \emph{Cohen-Macaulay and
Gorenstein complexes from a topological
point of view}, Preprint.

\bibitem[16]{nodecrep} D. Notbohm, \emph{Homology decompositions
for classifying spaces of
finite groups associated to modular representations}, Proc. LMS 64
(2001)

\bibitem[17]{nora2} D. Notbohm and N. Ray,
\emph{On Davis-Januszkiewicz homotopy types II; completion and localisation},
to appear in  Algebr. Geom. Topol.

\bibitem[18]{oliver} R. Oliver, \emph{Higher limits via Steinberg
representations}, Comm. in Alg. 22 (1994), 1381-1402

\bibitem[19]{quillen} D. Quillen,
\emph{The spectrum of an equivarariant
cohomolgy ring I}, Ann. Math. 94 (1971), 549, 572.

\bibitem[20]{rector} D. Rector, \emph{Noetherian cohomology rings
and finite loop spaces with torsion}, JPPA 32 (1984), 191-217.

\bibitem[21]{reisner} G. Reisner, \emph{ Cohen-Macaulay quotients of
polynomial rings}, Adv. Math. 21 (1976), 30-49.

\bibitem[22]{schwartz} L. Schwartz, \emph{Unstable modules over
the Steenrod algerba and Sullivan's fixed point set conjecture},
University of Chicago Press 1994.

\bibitem[23] {serre} J.-P. Serre, \emph{Alg\`ebre locale. Multiplic\'es},
SLNM 11, Berlin-New York 1965.

\bibitem[24]{weibel} C.A. Weibel, \emph{An introduction to
homological algebra}, Cambridge Stud. Adv Math. 38, Cambridge
Univ. Press 1997

\end{thebibliography}
\end{document}